\documentclass{article}
\usepackage [utf8] {inputenc}
\usepackage{mathtools}
\usepackage{amsthm}

\newtheorem*{theorem}{Theorem}
\newtheorem*{corollary}{Corollary}

\title{On finite groups with exactly one noncommutator}
\author{Saveliy V. Skresanov}
\date{}

\begin{document}
\maketitle

\begin{abstract}
	An element $x$ of a group $G$ is a \emph{commutator} if it can be expressed in the form $x = a^{-1}b^{-1}ab$
	for some $a, b \in G$. In 2010 MacHale posed the following problem in the Kourovka notebook: does there exist
	a finite group $G$, with $|G| > 2$, such that there is exactly one element of $G$ which is not a commutator?
	We answer this question in the affirmative and provide an infinite series of such groups, the smallest group
	in our construction having size $16609443840$.
\end{abstract}


\section{Introduction}

The study of commutators in finite groups has a long history, see, for instance, a survey~\cite{kappe}.
It is well known that the derived subgroup of a finite group need not consist solely of commutators;
the smallest such example having size~\( 96 \), see~\cite{gural}. One can construct a perfect group with a noncommutator,
see~\cite{isaacs} and an improvement in~\cite{gural2}. In contrast, by the positive solution to the Ore conjecture~\cite{ore},
every element of a nonabelian finite simple group is a commutator.

In the 20th edition of the Kourovka notebook~\cite{kourovka} MacHale asked the following question.
\medskip

\noindent
\textbf{Question 17.76.} \emph{Does there exist a finite group \( G \), with \( |G| > 2 \),
such that there is exactly one element in \( G \) which is not a commutator?}
\medskip

Let \( G \) be a finite group with exactly one noncommutator. Since every element outside the derived subgroup \( [G, G] \)
is not a commutator, we must have \( |G \setminus [G,G]| \leq 1 \). Hence either \( G \) is a cyclic group of order~\( 2 \),
which is excluded from MacHale's question, or \( G \) is perfect. If \( t \) is the noncommutator element of \( G \),
then \( t^{-1} \) and \( g^{-1}tg \), \( g \in G \), are also noncommutators. It follows that \( t^2 = 1 \) and \( t \)
lies in the center \( Z(G) \) of~\( G \).

Using GAP~\cite{gap}, Staroletov~\cite{staroletov} showed that there are no groups with one noncommutator among perfect groups of order at most \( 10^6 \),
with the possible exception of some group orders which were not present in the GAP perfect groups library at the time.
He also showed that there are no such examples among finite quasisimple groups,
i.e.\ among groups \( G \) such that \( G/Z(G) \) is a nonabelian simple group (see also~\cite[Table~1]{liebeck}).

We show that there is, in fact, a perfect group with exactly one noncommutator, so MacHale's question can be answered in the affirmative.
Our construction is inspired by a theorem of Isaacs~\cite{isaacs}.
He showed that if \( A \) is an abelian group and \( M \) is a nonabelian group, then the derived subgroup of the wreath product \( A \wr M \),
where \( M \) is assumed to act regularly, contains a noncommutator under some conditions (which were substantially relaxed by Guralnick in~\cite{gural2}).
In our construction, let \( M \) be the Mathieu group \( M_{11} \) acting transitively on \( 22 \) points; the point stabilizer
will be isomorphic to the alternating group of degree~\( 6 \). We consider the wreath product \( H = C_2 \wr M \) of the cyclic group of order~\( 2 \)
and \( M \). Note that \( M \) no longer acts regularly and \( |H| = 2^{22} \cdot |M| = 2^{22} \cdot 7920 \). Finally, we set \( G_{MacHale} = [H, H] \).

\begin{theorem}
	The group \( G_{MacHale} \) has order \( 16609443840 \) and contains exactly one noncommutator.
\end{theorem}

The proof of this theorem relies on a computation in GAP, but because of the size of the group a direct ``bruteforce'' computation is not feasible.
In the next section we describe how one can check that \( G_{MacHale} \) has one noncommutator in reasonable time.

The main result implies that there is an infinite series of groups with one noncommutator.

\begin{corollary}
	There are infinitely many finite groups with exactly one noncommutator.
\end{corollary}
\begin{proof}
	Let \( G \) be a finite nonabelian group with exactly one noncommutator~\( t \).
	Let \( T \) be a normal subgroup of \( G \times G \) generated by \( (t, t) \), and set \( K = (G \times G)/T \).
	Notice that \( |K| > |G| \). We will prove that \( K \) also has exactly one noncommutator.

	Suppose that \( (t, 1)T \in K \) is a commutator. Then there exist elements \( a, b \in G \times G \)
	such that \( (t, 1) \in [a, b]T \). If \( a = (a_1, a_2) \) and \( b = (b_1, b_2) \), then
	\( (t, 1) = ([a_1, b_1], [a_2, b_2])T \). Hence either \( t = [a_1, b_1] \), which is impossible since \( t \)
	is not a commutator in \( G \), or \( 1 = [a_2, b_2]t \). But that implies \( t = [a_2, b_2] \) which is again a contradiction.
	Therefore \( (t, 1) \) is not a commutator in \( K \).

	Let \( (x, y)T \in K \) be an element distinct from \( (t, 1)T \).
	We may assume that \( x, y \neq t \). Indeed, if \( x = t \), then \( y \neq 1 \), since \( (x, y) \not\in (t, 1)T \).
	So we may replace the representative \( (x, y) \) by \( (xt, yt) = (1, yt) \), where \( 1, yt \neq t \).
	Similar reasoning applies in the case when \( y = t \). Now, if \( x, y \neq t \), there exist elements \( a_1, a_2, b_1, b_2 \in G \)
	such that \( x = [a_1, b_1] \), \( y = [a_2, b_2] \), and hence \( (x, y)T = [a, b]T \) is a commutator in \( K \),
	where \( a = (a_1, a_2) \), \( b = (b_1, b_2) \).

	Given a nonabelian group \( G \) with one noncommutator we constructed a larger group \( K \) with one noncommutator.
	By applying this procedure repeatedly to the original group \( G_{MacHale} \), we obtain an infinite series of groups with one noncommutator.
\end{proof}

It follows immediately from the Corollary and the compactness principle that there is an infinite group of bounded exponent with exactly one noncommutator.

We would like to end this section with an open question.
It is not clear whether \( G_{MacHale} \) is the smallest nonabelian group with one noncommutator,
so it would be interesting to determine the smallest group with such property.
\medskip

\noindent
\textbf{Question.}
\emph{What is the smallest nonabelian finite group with exactly one noncommutator?}

\section{How to check properties of \( G_{MacHale} \)}

We gave an explicit construction of \( G = G_{MacHale} \) in the previous section, but for GAP computations we can
represent it as a permutation group of degree \( 44 \) generated by \( 4 \) elements:
\begin{align*}
	\tt G := Group([&(1,2)(43,44),\\
	&(1,2)(21,22),\\
	&(1,39,13,43,25)(2,40,14,44,26)(3,37,15,41,27)(4,38,16,42,28)\\
	&(5,35,23,17,31)(6,36,24,18,32)(7,33,21,19,29)(8,34,22,20,30),\\
	&(1,23,27,11,41)(2,24,28,12,42)(3,22, 26,10,44,4,21,25,9,43)\\
	&(5,16,20,40,32,6,15,19,39,31)(7,13,17,37,29)(8,14,18,38,30) ]);
\end{align*}

It can be easily checked in GAP that this group has order \( 16609443840 \) and its center has order \( 2 \).
Let \( t \) be the nonidentity element of the center.

We need to verify that all elements of \( G \) besides \( t \) are commutators.
Our group has \( 1280 \) conjugacy classes and it suffices to check whether a representative of each conjugacy class (besides \( t \))
is a commutator. In the ancillary file published alongside this preprint we provide GAP code which
defines a list \texttt{commPairs} containing \( 1279 \) pairs of permutations. Each pair \( a, b \) defines a commutator \( [a, b] \)
and it can be checked that these commutators give representatives of \( 1279 \) conjugacy classes of \( G \).

The supplementary GAP file also contains a function \texttt{CheckCommutators} which can be used to check this property; it takes the author's computer
about \( 15 \) minutes to perform this check.
The list \texttt{commPairs} was found by generating random commutators until all conjugacy classes (besides \( t \)) were covered.

Finally, in order to check that \( t \) is not a commutator, note that \( t \) acts on the set of conjugacy classes of \( G \) by left multiplication.
Indeed, if \( g^G \) is a conjugacy class, then \( t \cdot g^G = (tg)^G \) is also a conjugacy class. Moreover, \( t \) is a commutator if and only if
it stabilizes some conjugacy class: if \( t = g^{-1}x^{-1}gx \) for some \( x, g \in G \), then \( t \cdot g^G = g^G \) and vice versa.
This property can be easily tested in GAP by considering all conjugacy classes of \( G \); in the supplementary file we provide a function \texttt{CheckCentral}
which performs this check.

We note that the character table of \( G \) determines its set of commutators, see, for instance,~\cite[Problem~3.10]{char}.
It is possible to compute the full character table of \( G \) in GAP and verify the required properties, but this computation
takes much longer than the method proposed above.

\section{Acknowledgements}

The author would like to thank A.M.~Staroletov for fruitful discussions.
The research was carried out within the framework of the Sobolev Institute of Mathematics state contract.

\bigskip

\noindent
\emph{Saveliy V. Skresanov}

\noindent
\emph{Sobolev Institute of Mathematics, 4 Acad. Koptyug avenue,\\ 630090 Novosibirsk, Russia}

\noindent
\emph{Email address: skresan@math.nsc.ru}

\end{document}